\documentclass[11pt]{article}

\usepackage{amssymb,bbm}
\usepackage{amsfonts}
\usepackage{latexsym}
\usepackage{amsmath}
\usepackage{amsthm}
\usepackage{geometry} 
\usepackage{graphicx}
\usepackage{verbatim}
\usepackage[colorlinks=true]{hyperref} 

\usepackage{graphicx}
\usepackage{amsmath,amsthm,amsfonts,amssymb, nccmath}
\usepackage[lofdepth,lotdepth]{subfig}
\newtheorem{proposition}{Proposition}
\newtheorem*{definition*}{Definition}
\newtheorem{lemma}{Lemma}[section]
\newtheorem{remark}{Remark}

\usepackage{tabularx,booktabs}
\newcolumntype{C}{>{\centering\arraybackslash}X} 
\usepackage{color}
\usepackage{comment}
\usepackage{gensymb}
\usepackage{bigints}
\usepackage[linesnumbered,lined,boxed,commentsnumbered, ruled]{algorithm2e}

\def\lddots{\mathinner{\mkern1mu\raise1pt\hbox{.}\mkern2mu  
\raise4pt\hbox{.}\mkern2mu\raise7pt\vbox{\kern7pt\hbox{.}}\mkern1mu}}
\makeatletter
\def\numberbysection{\@addtoreset{equation}{section}
 \def\theequation{\thesection.\arabic{equation}}}
\makeatother  
  
\numberbysection

\newcommand{\be}{\begin{eqnarray}}  
\newcommand{\ee}{\end{eqnarray}}

 \title{\bf Analytic inversion of a Radon transform on double circular arcs with applications in Compton Scattering Tomography}
\author{ \textsf{C\'ecilia Tarpau$^{a,b,c}$,}
\textsf{Javier Cebeiro$^d$,}
\textsf{Mai K. Nguyen$^a$,}
\textsf{Genevi\`eve Rollet$^b$}
\textsf{and Marcela Morvidone$^d$}
\\
\textit{\footnotesize $^{a}$ETIS (CNRS UMR 8051), Universit\'e de Cergy-Pontoise, F-95302 Cergy-Pontoise, France} 
\\
\textit{\footnotesize $^{b}$ LPTM (CNRS UMR 8089),Universit\'e de Cergy-Pontoise, F-95302 Cergy-Pontoise, France} 
\\
\textit{\footnotesize $^{c}$ LMV (CNRS UMR 8100),Universit\'e de Versailles Saint-Quentin, F-78035 Versailles, France}
\\ 
\textit{\footnotesize $^{d}$ CEDEMA, Universidad Nacional de San Mart\'in, B-1650 San Mart\'in, Argentina} \\}
\date{}

\clearpage
\begin{document}

\maketitle
\thispagestyle{empty}
\abstract{{\footnotesize In this work we introduce a new Radon transform which arises from a new modality of Compton Scattering Tomography (CST). This new system is made of a single detector rotating around a fixed source.  Unlike some previous CST, no collimator is used at the detector. Such a system allows us to collect scattered photons coming from two opposite sides of the source-detector segment, hence the manifold of the associated Radon transform is a family of double circular arcs. As first main theoretical result, an analytic inversion formula is established for this new Radon transform. This is achieved through the formulation of the transform in terms of circular harmonic expansion satisfying the consistency conditions in Cormack's sense. Moreover, a fast and efficient numerical implementation via an alternative formulation based on Hilbert transform is carried out. Simulation results illustrate the theoretical feasibility of the new system.  From a practical point of view, an uncollimated detector system considerably increases the amount of collected data, which is particularly significant in a scatter imaging system. } }

\section{Introduction}
\label{sec:introduction}
The seminal works of Radon \cite{radon17} and Cormack \cite{cormack63} on the classical Radon transform on straight lines were followed by several studies attempting to generalize this integral transform to various manifolds in two and three dimensions. Many of these generalizations are formulations in either full or partial circular paths, hence they are called Circular Radon Transforms (CRTs) and Circular Arc Radon Transforms (CARTs) respectively. A major part of these works concerns inversion formulas of CRTs. One can refer for example to the work of Cormack on the family of circles passing through the origin \cite{cormack81, cormack84}, the works of Ambartsoumian \cite{ambartsoumian2010inversion} and Haltmeier \cite{haltmeier2007thermoacoustic} on circles centered on a circle for thermoacoustic and photoacoustic tomography or the work of Redding \cite{redding2001inverting} about a CRT with applications in synthetic aperture radar imaging. There is also the work of Palamodov \cite{palamodov2000halfcircles} on an inversion formula to reconstruct objects from their data over a family of half-circles with applications in seismic tomography and synthetic aperture radar. About Radon transform on circular arcs, Nguyen and Truong \cite{nguyen10, truong2011radon} proposed the inversion Radon transforms on different families of circular arcs modelling data acquisition and image reconstruction of new modalities of Compton Scattering Tomography (CST). Syed \cite{syed2016numerical} proposed a numerical inversion for circular arcs with a fixed angular span with applications in photoacoustic tomography. Other aspects of CRTs and CARTs are also studied such as injectivity and range conditions. The reader can refer respectively to \cite{Ambartsoumian_2005, agranovsky1996injectivity} and \cite{agranovsky2007range, ambartsoumian2006range, finch2006range} for these issues.  

In this paper we introduce a new Radon transform on a family of double circular arcs. This study is motivated by the proposition of a new Compton scattering tomography. We will now the general functioning principle of CST modalities.  

\subsection{Compton Scattering Tomography (CST)}
In Compton scattering tomography, the objective is to use radiation scattered according to Compton effect to reconstruct the electron density map of the object. This type of imaging has been early studied especially by Lale \cite{lale1959examination}, Clarke \cite{clarke1969compton} and by Farmer \cite{farmer1971new}. The reader can refer to \cite{livre} for a more detailed history on CST systems.

Conventional Computed Tomography (CT) uses only primary (transmitted) radiation and considers scattered radiation as noise. CST is a \textit{novel} approach that takes advantage of scatter radiation to scan objects. Images generated in that way are claimed to offer better performance in medical imaging particularly to identify lung tumors \cite{jones2018characterization, redler2018compton}. In fact, if a photon emitted by a source with energy $E_0$ finds an electron on its path inside matter, this photon is scattered subtending an angle $\omega$with its original direction. Its energy $E(\omega)$ after collision is given by the Compton formula
\begin{equation}
     E(\omega) = \frac{E_0}{1+\frac{E_0}{m\,c^2}\left(1-\cos(\omega)\right)},
\end{equation}
with $m$ the electron mass and $c$ the speed of light. 
When a single scattered photon outgoing from the object with an energy $E(\omega)$ is collected by a detector, this one-to-one correspondence between energy and angle ensures that the photon was scattered on a circular arc. This circular arc passes through the source and the detector and is labeled by the scattering angle $\omega$. 
Consequently, assuming Compton effect is the only source of energy attenuation for emitted radiation and considering only first order scattering, the modelling of data measurement with modalities of Compton scatter tomography leads to CARTs on different families of circular arcs according to the chosen geometry of the system. These assumptions are the basis for mathematical tractability \cite{nguyen10, webber2018three, webber2019microlocal, cebeiro, truong2019compton, tarpau19trpms, tarpau2020compton, cebeiro18, cebeiro2017new} and some of their practical implications are discussed at the end of the paper. The inversion of these Radon transforms represents the theoretical challenge raised by CST modalities, required for image reconstruction. 

The first CST system whose data acquisition has been modelled by a Radon transform was proposed by Norton \cite{norton94}. This modality is made of a fixed source and a line of detectors. Image acquisition is performed on half circular arcs having a fixed common end-point in the source and the other ends on the straight line of detectors. The reader can also refer to \cite{num_inv} where Rigaud gave another numerical inversion method for the Radon transform associated to Norton's modality. The second CST modality, proposed by Nguyen and Truong \cite{nguyen10, num_inv, rigaudIEEE13}, is composed of a pair source - detector diametrically opposed on a circle. This system has lead to a CART based on circular arcs having a chord of fixed length. A third modality has been proposed recently by us \cite{tarpau19trpms, tarpau2020compton, cebeiro18}. This modality is composed of a fixed source and a detector array on a ring. This one has led to a new CART having a fixed common end-point in the source and the other ends on the detector ring. A system with a similar geometry has been studied by Rigaud \cite{rigaud_rotation_free2017}. He proposed in this work a reconstruction algorithm of the singular support of the object under study.

Three dimensional CST systems have also been proposed. The reader can refer to the works of Webber on the extension of the Nguyen and Truong's system \cite{webber2018three} and on the proposition of a system having a translational geometry \cite{webber2020compton} or to the work of Rigaud \cite{rigaud20183d} who proposed several potential 3D modalities associated to a general class of toric Radon transforms. 

\subsection{A new CST system with a collimation-free detector }

\label{sec:A_new_CST}

The context of this study is the introduction of a new two-dimensional CST system. This modality is based on a fixed source $S$, assumed to be monochromatic and a detector $D$ moving on a circle of radius $R$ around the source (see Fig. \ref{fig:paramCST4}), and localized by its angular position $\varphi$. Hence, $D$ can be defined by its Cartesian coordinates as $D(\varphi) = R\,(\cos{\varphi}, \sin{\varphi}).$
The object to scan is placed outside of this circle. In this two-dimensional setup, a plate collimator in the source restricts emitted photons to the plane $(x,y)$. Thus, cross-sections of the object are scanned. In contrast with other two-dimensional CST designs \cite{nguyen10,tarpau19trpms,rigaud_rotation_free2017,webber2020compton}, no collimator is required at the detector. This notable feature enables an increase of the amount of acquired data for a given position of the detector and thus a possible reduction in the acquisition time.
In this setting, for an angular position $\varphi$ of the detector a scattering angle $\omega$ corresponds to two circular arcs. So, modelling of data acquisition with this collimation-free detector leads to a Radon transform on double circular arcs (DCART). 
Besides the uncollimated detector, this modality has other material advantages such as not requiring relative movement between the system and the object. In addition, the usage of a circular detector path enables reduction of the system size: with a linear detector, information corresponding to large scanning circles is recorded far from the source while with the proposed geometry photons are always collected at a fixed distance of it.

\subsection{Objectives and outline of the paper}
Image reconstruction from projections obtained with this modality requires inversion of the Radon transform on double circular arcs. In this paper, we derive the analytic inversion of this integral transform in order to reconstruct cross-sectional images of an object.  
This paper is outlined as follows.  Section II introduces the measurement model of the CST system as well as the corresponding forward DCART. Section III presents the whole procedure for inverting the DCART. Section IV deals with the discrete formulations for the forward and inverse DCART used for simulations. Numerical simulations in section V illustrate the theoretical feasibility of the new system and the performance of the proposed reconstruction algorithm. Section VI opens a discussion about future works ranging from practical and physical considerations to the possibility of an extension in three dimensions. Concluding remarks of section VII summarize and end the paper.  Technical details in demonstrations are presented in appendices A and B.


\section{Measurement model of the proposed CST system}

\begin{figure}[!ht]
    \centering
    \includegraphics[scale=0.78]{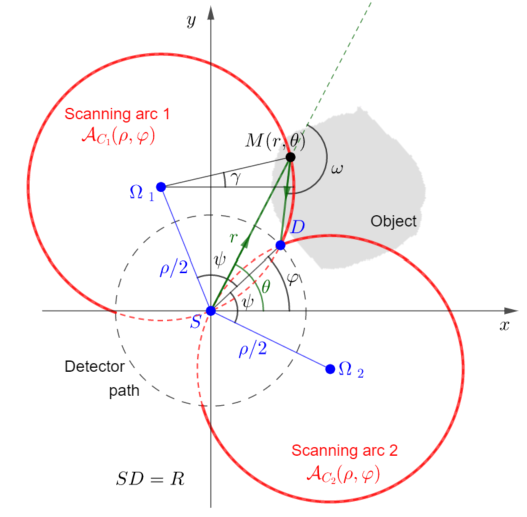}
    \caption{\footnotesize Setup and parametrisation of the new CST modality. Black dotted circle: detector path, Red continuous curves: scanning double arcs, Red dotted curves: portion of circles not used for acquisition}
    \label{fig:paramCST4}
\end{figure}

Acquisitions using this system are made on a family of double  circular arcs $\mathcal{A}_{C_1}$ and $\mathcal{A}_{C_2}$ of centre $\Omega_1$ and $\Omega_2$ whose union is denoted $\mathcal{D}$ in the rest of the paper: $\mathcal{D}(\rho, \varphi) = \mathcal{A}_{C_1}(\rho, \varphi)\cup\mathcal{A}_{C_2}(\rho, \varphi)$. 
This family of double circular arcs is then defined relatively to $\rho$ the diameter of scanning circles and $\varphi$ the angular position of the detector  (see Fig. \ref{fig:paramCST4}). In fact, using $\rho$ or $\omega$ in order to define the family of double circular arcs is equivalent, since they are related by $\rho = R/\sin{\omega}$, where $\omega \in \left[\pi/2, \pi\right[$.
Given $\varphi$ the angular position of $D$ and $\omega$ the considered scattering angle, polar equations of $\mathcal{A}_{C_1}$ and $\mathcal{A}_{C_2}$ are defined as 
\begin{equation}
    \mathcal{A}_{C_m}(\rho, \varphi): r = \rho\, \cos\left(\theta-\left(\varphi-(-1)^{m}\psi\right)\right), \; m\in\{1,2\}
\end{equation}
where $\theta\in[\varphi, \varphi+2\omega-\pi]$ for $\mathcal{A}_{\mathcal{C}_1}$, $\theta\in[\varphi-2\omega+\pi, \varphi]$ for $\mathcal{A}_{\mathcal{C}_2}$ and $\psi = \cos^{-1}{(R/\rho)}$. 

Data measurement using this CST system is modelized by the generalized Radon transform $\mathcal{R}_\mathcal{D}$ on the family of double circular arcs $\mathcal{D}$ whose definition is: 
\begin{definition*}
Let $f$ be an unknown function, non negative, continuous and compactly supported outside the disc of radius $R$ centered at the origin. The Radon transform on double circular arcs $\mathcal{R}_\mathcal{D}$ maps $f$ into the set of its integrals over the family of double circular arcs $\mathcal{D}$ as \begin{equation}
    \mathcal{R}_\mathcal{D}f(\rho, \varphi) = \int_{\mathcal{D}(\rho, \varphi)} f(r,\theta)\, ds,
    \label{eq:formation}
\end{equation}
where $ds$ is the elementary arc length measure on the considered double circular arc (s denotes the curvilinear abscissa).
\end{definition*}

\noindent

In a physical sense, function $f$ is a cross-sectional map of the electronic density of the object. Thus, when the object is scanned as explained in section \ref{sec:A_new_CST}, the flux of photons of energy $E(\omega)$ registered by the detector at site $(R \cos \varphi, R \sin \varphi)$ is proportional to the integral of function $f$ on the pair of circular arcs labelled by $\mathcal{D}(\rho, \varphi)$, i.e. $\mathcal{R}_\mathcal{D}f$. This dependence is weighted by independent factors that model some physical effects. Thus, measurements may be corrected prior to reconstruction by applying the appropriate factors \cite{rigaud2012novel}. For instance, the Klein-Nishina cross-sectional probability depends on $\rho$ (i.e. $\omega$) and accounts for different probabilities of scattering for different angles. 

\section{An analytic inversion formula for the DCART}

In this section, we derive the main equations of the procedure for inverting the DCART. First, we use a circular harmonic expansion to establish a relation between the components of a function $f$ and $\mathcal{R}_\mathcal{D}f$. Then, with a similar approach as Cormack's one, we obtain \eqref{inv_gn}, which is the inversion formula in circular harmonic expansion. Finally, a closed formulation \eqref{eq:recons_f}, which will be used in numerical simulations, is obtained.
\subsection{Circular harmonic expansion}
Functions $f(r,\theta)$ and $\mathcal{R}_\mathcal{D}f(\rho, \varphi)$ are expanded in terms of Fourier series, where $f_n(r)$ and $(\mathcal{R}_\mathcal{D}f)_n(\rho)$ are respectively their circular harmonic expansion components 
\begin{align}
    f(r,\theta) &= \sum_{n=-\infty}^{\infty}\, f_n(r)\;e^{in\theta}\label{eq:decomp_f}\\
    (\mathcal{R}_\mathcal{D}f)(\rho, \varphi) &= \sum_{n=-\infty}^{\infty} \, (\mathcal{R}_\mathcal{D}f)_n(\rho)\;e^{in\varphi}
    \label{eq:decomp_g}
\end{align}

\noindent
where
\begin{align}
    f_n(r) &= \frac{1}{2\pi}\,\int_0^{2\pi} f(r, \theta)\,e^{-in\theta}\, d\theta\label{eq:recomp_f}\\ 
    (\mathcal{R}_\mathcal{D}f)_n(\rho) &= \frac{1}{2\pi}\,\int_0^{2\pi} \,\mathcal{R}_\mathcal{D}f(\rho, \varphi)\,e^{-in\varphi}\,d\varphi
    \label{eq:recomp_g}
\end{align}

\subsection{Inversion formula in circular harmonics expansion}
\begin{proposition}
Function $f$ is completely recovered via its circular harmonic expansion $f_n$ from circular expansion of data measurement $(\mathcal{R}_\mathcal{D}f)_n$ with 
\begin{equation}
        f_n(r) = \frac{1}{\pi}\frac{d}{dr}\int_R^r  \frac{\cosh\left(n\cosh^{-1}{\left(\frac{r}{\rho}\right)}\right)}{\rho\sqrt{\left(\frac{r}{\rho}\right)^2-1}}\frac{(\mathcal{R}_{\mathcal{D}}f)_n(\rho)}{ 2\cos{\left(n\cos^{-1}{\frac{R}{\rho}}\right)}}d\rho
        \label{inv_gn}
\end{equation}
\end{proposition}
\begin{proof}
Data projection on double circular arcs $\mathcal{R}_\mathcal{D}f(\rho, \varphi)$ can be decomposed on two Radon transforms on the families of circular arcs $\mathcal{A}_{\mathcal{C}_1}(\rho, \varphi)$ and $\mathcal{A}_{\mathcal{C}_2}(\rho, \varphi)$. Consequently, we denote $\mathcal{R}_{\mathcal{A}_{\mathcal{C}_1}}(\rho, \varphi)$ and $\mathcal{R}_{\mathcal{A}_{\mathcal{C}_2}}(\rho, \varphi)$ the Radon transforms on respectively $\mathcal{A}_{\mathcal{C}_1}(\rho, \varphi)$ and $\mathcal{A}_{\mathcal{C}_2}(\rho, \varphi)$, whose sum gives $\mathcal{R}_\mathcal{D}f$. On a similar way, $\mathcal{R}_{\mathcal{A}_{\mathcal{C}_1}}(\rho, \varphi)$ and $\mathcal{R}_{\mathcal{A}_{\mathcal{C}_2}}(\rho, \varphi)$ can be decomposed in Fourier series to obtain respectively  $(\mathcal{R}_{\mathcal{A}_{\mathcal{C}_1}}f)_n$ and $(\mathcal{R}_{\mathcal{A}_{\mathcal{C}_2}}f)_n$. By linearity, we have $(\mathcal{R}_\mathcal{D}f)_n(\rho) = (\mathcal{R}_{\mathcal{A}_{\mathcal{C}_1}}f)_n(\rho)  + (\mathcal{R}_{\mathcal{A}_{\mathcal{C}_2}}f)_n(\rho).$ First step of this demonstration is to make explicit the circular expansions of $\mathcal{R}_{\mathcal{A}_{\mathcal{C}_m}}f$, $m\in\{1,2\}$. 
An interesting property of $\mathcal{A}_{\mathcal{C}_1}$ and $\mathcal{A}_{\mathcal{C}_2}$ is their symmetry about $(\varphi+\psi)$ and $(\varphi-\psi)$ respectively (see Fig. \ref{fig:paramCST4}). This feature allows us to see each ${\mathcal{A}_{\mathcal{C}_{m}}}$ as a combination of two equal and symmetric elements of arcs. One of these elements of arc can be rewritten as an angular shift of the other one, hence the following relation for $\mathcal{R}_{\mathcal{A}_{\mathcal{C}_{m}}}f$
\begin{equation}
    \mathcal{R}_{\mathcal{A}_{\mathcal{C}_{m}}}f(\rho, \varphi) =\\ \sum_{n=-\infty}^{\infty}\int_{\mathcal{A}_{\mathcal{C}_{m}}^+(\rho,\varphi)}\, f_n(r)\left(e^{in\theta}+e^{in\left[2(\varphi-(-1)^{m}\psi)-\theta\right]}\right)ds_{m} \label{eq:sym}
\end{equation}
where $\mathcal{A}_{\mathcal{C}_{m}}^+$ denotes the half part of circular arc  $\mathcal{A}_{\mathcal{C}_{m}}$, \\$\theta \geq \varphi -(-1)^{m} \psi$ and $ds_{m}$ is the elementary arc length measure on $\mathcal{A}_{\mathcal{C}_{m}}$. 

Observing that 
\begin{equation}
    e^{in\theta}+e^{in\left[2+(\varphi-(-1)^{m}\psi)-\theta\right]}=\\2\,e^{in\varphi}\,e^{in\psi}\cos\left[n\left(\theta-\left(\varphi-(-1)^{m}\psi\right)\right)\right],
    \label{eq:aide_identification}
\end{equation}

\noindent and plugging \eqref{eq:aide_identification} into \eqref{eq:sym}, one obtains the circular harmonic expansions of $\mathcal{R}_{\mathcal{A}_{\mathcal{C}_1}}f$ and $\mathcal{R}_{\mathcal{A}_{\mathcal{C}_2}}f$
\begin{equation}
     \frac{(\mathcal{R}_{\mathcal{A}_{\mathcal{C}_{m}}}f)_n(\rho)}{2\,e^{-(-1)^{m} in\psi}} = \\ \int_{\mathcal{A}_{C_{m}}^+(\rho, \varphi)} f_n(r)\cos\left[n\left(\theta-\left(\varphi-(-1)^{m}\psi\right)\right)\right]ds_{m}.
    \label{eq:g2n}
\end{equation}
   
Straightforward computations show that 
$$\theta-\varphi-(-1)^{m} \psi = \cos^{-1}{\left(r/\rho\right)} \hbox{ and } ds_{m} = \left(1-\left(r/\rho\right)^2\right)^{-1/2}dr.$$

Equation \eqref{eq:g2n} becomes
\begin{equation}
    \frac{(\mathcal{R}_{\mathcal{A}_{\mathcal{C}_{m}}}f)_n(\rho)}{2\,e^{-(-1)^{m} i n\psi}}= \int_{R}^{\rho}f_n(r)\,
    \frac{\cos\left(n\cos^{-1} (r/\rho)\right)}{\sqrt{1-\left(r/\rho\right)^2}}dr.
    \label{eq:g_1n_g_2n}
\end{equation}


 Hence, from the addition of the expressions in $(\mathcal{R}_{\mathcal{A}_{\mathcal{C}_1}}f)_n$ and $(\mathcal{R}_{\mathcal{A}_{\mathcal{C}_2}}f)_n$ in \eqref{eq:g_1n_g_2n}, the connection between circular components of $f$ and $\mathcal{R}_\mathcal{D}f$ can be written
\begin{equation}
        \frac{(\mathcal{R}_\mathcal{D}f)_n(\rho)}{4\cos{\left(n\psi\right)}}= \int_{R}^{\rho} f_n(r)\,
    \frac{\cos\left(n\cos^{-1}(r/\rho)\right)}{\sqrt{1-\left(r/\rho\right)^2}}dr.
    \label{eq:g_n}
\end{equation}
  

Then, denoting 
\begin{equation}
    G_n(\rho) =  \frac{(\mathcal{R}_{\mathcal{D}}f)_{n}(\rho)}{2 \cos{(n\psi)}}\label{eq:Gn}
\end{equation}
\noindent and multipliying both sides of (\ref{eq:g_n})  by
$$ \int_R^t \frac{\cosh\left(n\cosh^{-1}{\left(t/\rho\right)}\right)}{\rho\sqrt{\left(t/\rho\right)^2-1}}d\rho,$$
\noindent with $t\in \mathbb{R}, t>R$, one gets

\begin{equation}
     \frac{1}{2}\int_R^t \frac{\cosh\left(n\cosh^{-1}{\left(t/\rho\right)}\right)}{\rho\sqrt{\left(t/\rho\right)^2-1}}G_{n}(\rho)d\rho = \int_R^t f_n(r)\int_{r}^{t} \frac{\cosh\left(n\cosh^{-1}{\left(t/\rho\right)}\right)}{\rho\sqrt{\left(t/\rho\right)^2-1}} \frac{\cos\left(n\cos^{-1}{(r/\rho)}\right)}{\sqrt{1-\left(r/\rho\right)^2}}d\rho dr,
     \label{eq:rearranging}
\end{equation}
where the right $\rho$-integral is $\pi/2$ \cite{cormack81}. Then, differentiating with respect to the variable $t$, one gets
\begin{equation}
    f_n(t) = \frac{1}{\pi} \frac{d}{dt}\int_R^t  \frac{\cosh\left(n\cosh^{-1}{\left(t/\rho\right)}\right)}{\rho\sqrt{\left(t/\rho\right)^2-1}}\,G_{n}(\rho)d\rho
\end{equation}

Going back to coefficients $(\mathcal{R}_\mathcal{D}f)_{n}$ and substituting  $t$ by $r$, one finds (\ref{inv_gn}). 
\end{proof}

\begin{remark}
Equation \eqref{inv_gn} demonstrates explicitly the Cormack's hole theorem: in order to determine $f(r,\theta)$ by its circular harmonic expansion $f_n(r)$, the knowledge of the coefficients $(\mathcal{R}_{\mathcal{D}}f)_n(\rho)$ in the annular domain $R<\rho<r$ is sufficient. 
\end{remark}

\subsection{A closed formulation of \eqref{inv_gn}}
\begin{proposition}
Denoting $G(\rho, \varphi)$ the function corresponding to the Fourier series expansion $\displaystyle \sum_{n=-\infty}^{\infty}\, G_n(\rho)\;e^{in\varphi}$ with $G_n$ as defined in \eqref{eq:Gn}, $f$ can be completely recovered from $G$ as follows
\begin{equation}
    f(r,\theta) =\frac{1}{2\pi^2 r} \int_0^{2\pi} \text{p.v.}\left\{\int_R^\infty  \frac{\partial G(\rho, \varphi)}{\partial \rho} \frac{\rho}{r-\rho\cos{(\theta-\varphi)}}d\rho\right\}d\varphi
    \label{eq:recons_f}
\end{equation}
\noindent where p.v. denotes the Cauchy principal value.
\end{proposition}

\begin{proof}
This result is achieved introducing consistency conditions \cite{cormack84, truong2011radon} in terms of Cormack sense. The complete demonstration is proposed in the Appendices. 
\end{proof}

\begin{remark}
Equation \eqref{eq:recons_f} can be rewritten using the Hilbert transform and implemented in a more efficient way using standard tools of discrete Fourier analysis, as we are going to show in the next section.
\end{remark} 


\section{Numerical inversion }
\subsection{Numerical formulation of the forward DCART}
A parametrisation in Cartesian coordinates (instead of equations in polar coordinates) is preferable to perform numerical simulations in order to have the same distance between adjacent running points on the considered scanning circle. Hence, a scanning arc $\mathcal{A}_{C_m}$, $m\in\{1,2\}$ can be seen as the shift of a circle centred at the origin of identical radius, to its center $\Omega_m(x_{\Omega_m},y_{\Omega_m})$ with a restriction of the domain of the variable $\gamma$ (see Fig. \ref{fig:paramCST4}) : 
\begin{equation}
    \mathcal{A}_{C_m}(\rho, \varphi): (x_m(\gamma), y_m(\gamma)) = (x_{\Omega_m}(\rho, \varphi),y_{\Omega_m}(\rho, \varphi)) + \frac{\rho}{2}(\cos{\gamma}, \sin{\gamma}), \gamma\in[\gamma_{m_{min}}, \gamma_{m_{max}}]
\end{equation}

Cartesian parametrisation of  $\mathcal{A}_{C_m}$ is 
\begin{equation}
    \mathcal{A}_{C_{m}}(\rho, \varphi): (x_{m}(\gamma),y_{m}(\gamma)) = \frac{\rho}{2}\left(\cos{\left(\varphi-(-1)^m\psi\right)}+\cos{\gamma}, \sin{\left(\varphi-(-1)^m\psi\right)}+\sin{\gamma}\right)
\end{equation}

\noindent where $ \gamma \in \left[\varphi-\psi, \varphi+3\psi\right]$ for $\mathcal{A}_{C_1}$ and $ \gamma \in \left[\varphi-3\psi, \varphi+\psi\right]$ for $\mathcal{A}_{C_2}$.
Then, the numerical computation of the forward DCART requires a discrete version of it. This process is achieved with a linear interpolation to make coincide the position of the object with the Cartesian parametrisation of the double family of circular arcs and an approximation of the integral with a sum according to the equation
\begin{equation}
    \mathcal{R}_{\mathcal{D}}f(\rho_i, \varphi_j) = \frac{\rho_i}{2}\Delta_\gamma \sum_{\substack{\gamma_k\in[\gamma_{1_{min}},\gamma_{1_{max}}]\\\\cup[\gamma_{2_{min}},\gamma_{2_{max}}]}}f\left(x(\gamma_k), y(\gamma_k)\right)
    \label{eq:model_algo}
\end{equation}
where $\Delta_\gamma$ is the sampling angular distance of $\gamma$, $\rho_i$ and $\varphi_j$ are the discrete versions of $\rho$ and $\varphi$ respectively: $\rho_i = i(\rho_{max}-R)/N_\rho$, $i = 1, ..., N_\rho$ and $\varphi_j = j\cdot(2\pi/N_\varphi)$, $j=1, ..., N_\varphi$. Hence, $\mathcal{R}_{\mathcal{D}}f(\rho_i, \varphi_j)$ is a $N_\rho\times N_\varphi$ matrix. $\rho_{max}$ refers to the largest diameter of the scanning double circular arcs. Fig. \ref{fig:ex_acq_shepp} shows an example of data measurement for the Shepp-Logan phantom. 
\begin{figure}
    \centering
    \includegraphics[scale = 0.15]{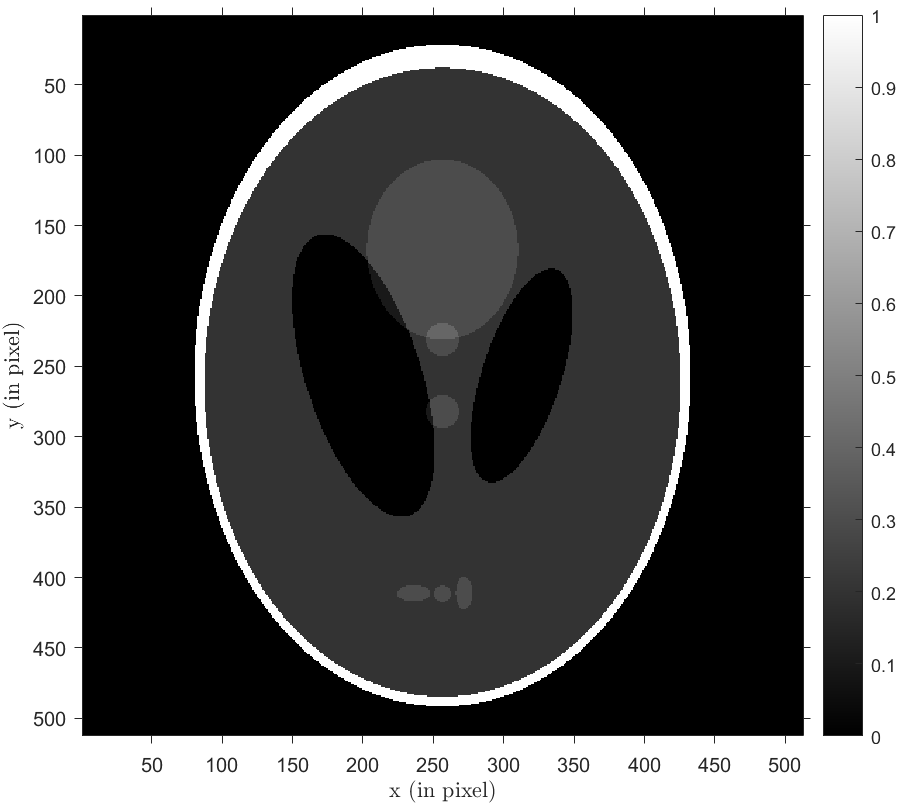} \hspace{2cm}
    \includegraphics[scale = 0.15]{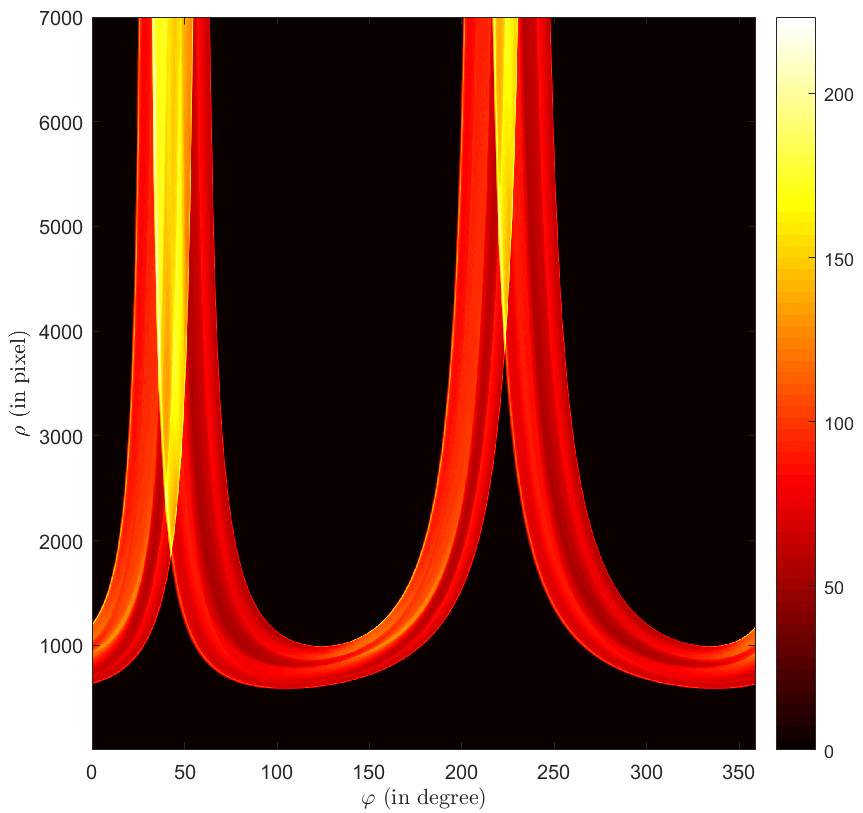}
    \caption{\footnotesize Left: Object. Right: corresponding DCART for $R=256$, $\rho_{max} = 7000$, $N_\rho = 6744$ and $N_\varphi=1609$.}
  \label{fig:ex_acq_shepp}
\end{figure}
\subsection{Reconstruction strategy}
For image reconstruction, we use the Hilbert transform, related to the Cauchy principal value as
\begin{equation}
     \mathcal{H}\{u\}(t) = \frac{1}{\pi} \text{ p.v.}\left\{\int_{-\infty}^{\infty}\frac{u(\tau)}{t-\tau}d\tau\right\}.
\end{equation} 

The Hilbert transform can be computed in the Fourier domain, where we have :
\begin{equation}
        \mathcal{H}\left\{u\right\}(t)= \mathcal{F}^{-1}(-i\cdot\hbox{sign}(\nu)\cdot\mathcal{F}(u)(\nu))(t),
        \label{eq:hilbert_fourier}
    \end{equation}
    where $\mathcal{F}$ denotes the one-dimensional Fourier transform. 
    
Equation (\ref{eq:recons_f}) of image reconstruction becomes consequently
\begin{equation}
    f(r,\theta) = \frac{1}{2\pi r}\cdot\int_0^{2\pi} \frac{1}{\cos(\theta-\varphi)} \mathcal{H}\left\{\frac{\partial G(\rho, \varphi)}{\partial\rho}\cdot\rho\right\}\left(\frac{r}{\cos{(\theta-\varphi)}}\right)d\varphi.
    \label{eq:recons2}
\end{equation}


Finally, using the correspondence between polar coordinates $(r,\theta)$ and Cartesian coordinates $(x,y)$ and  $r \cos{(\theta-\varphi)}= x\cos{\varphi}+y\sin{\varphi}$, image reconstruction equation used for simulation is 
\begin{equation}
    f(x,y) = \frac{1}{2\pi}\int_0^{2\pi} \frac{1}{x\cos{\varphi}+y\sin{\varphi}}\cdot\mathcal{F}^{-1}\left(-i\cdot\text{sign}(\nu)\mathcal{F}\left(\frac{\partial G(\rho, \varphi)}{\partial\rho}\rho\right)(\nu)\right)\left(\frac{x^2+y^2}{x\cos{\varphi}+y\sin{\varphi}}\right)d\varphi.
    \label{eq:recons_cart_final}
\end{equation}

The projections $G(\rho, \varphi)$ are computed via the circular harmonic components of $\mathcal{R}_{\mathcal{D}}f(\rho, \varphi)$ with (\ref{eq:Gn}). However, zeros in the denominator may be source of instability and regularization may be required. According to \cite{moon2017analytic}, we add a regularization parameter $\epsilon$ in \eqref{eq:Gn} (equal to $1$ in the proposed simulation) in order to compute the circular harmonic components of $G$: 
\begin{equation}
    G_n(\rho) = \frac{\cos{(n\psi)}}{\epsilon^2+\cos{(n\psi)}^2}\, \frac{(\mathcal{R}_\mathcal{D}f)_n(\rho)}{2}. 
    \label{eq:composeGn}
\end{equation}

Algorithm \ref{algo:reconssss} summarizes the different steps for reconstructing the object from \eqref{eq:recons_cart_final}.

\IncMargin{1em}
    \begin{algorithm}[!ht]
    \footnotesize
        \SetAlgoLined
        \KwData{$\mathcal{R}_\mathcal{D}f(\rho, \varphi)$, projections on double circular arcs of function $f$}
        \KwResult{$f(x,y)$}
        Compute circular harmonic expansion of $\mathcal{R}_\mathcal{D}f(\rho, \varphi)$ to compute $G_n(\rho)$ with \eqref{eq:composeGn} and recompose $G(\rho, \varphi)$\;
        Compute discrete derivation of $G(\rho, \varphi)$ relative to variable $\rho$ and multiply the result by $\rho$\;
        Write the Hilbert transform as a filtering operation in Fourier domain using \eqref{eq:hilbert_fourier}\; 
        For each $\varphi$, interpolate the data on the considered scanning circles $(x^2+y^2)/(x\cos\varphi+y\sin\varphi)$ of \eqref{eq:recons_cart_final}\;
        Weight the result using the factor $1/(x\cos\varphi+y\sin\varphi)$\;
        Sum the weighted interpolations on all directions $\varphi$\;
        Weight the result by $\frac{1}{2\pi}$\;
    \caption{\footnotesize Reconstruction of object}
    \label{algo:reconssss}
    \end{algorithm}
    \DecMargin{1em}

\section{Experiments and study of the performance of the reconstruction algorithm}
\subsection{General parameter choices}
We will perform simulations on three different phantoms, the Shepp-Logan, the Derenzo (also called Jaszczak) and bars phantoms. These phantoms will allow us to evaluate different criteria such as spatial resolution, contrast and the ability of the proposed algorithm to reconstruct singularities tangent to lines with arbitrary slopes. 
In all proposed simulations, the size of the object is $512\times512$ pixels. The detector is moving on a ring of radius $R=256$ pixels with a constant step of arc length between two adjacent positions. This represents an amount of $N_\varphi=2\pi R = 1609$ different positions for the detector to collect data. We know the parameter choices should satisfy the condition $N_\rho\times N_\varphi \geq N\times N$ according to \cite{RB1990}, where $N_\rho$ is the number of double circular arcs per detector position. To quantitatively asses the quality of reconstructions, we use NMSE $= ||f-f_0||_2^2/N^2$ and NMAE  $= ||f-f_0||_1/N^2$ metrics, where $f_0$ and $f$ are the respective original and reconstructed objects and $||.||_1$ and $||.||_2$ refer respectively to the $1$ and $2$-norm. These results are summed up in Table \ref{tab:NMSE_NMAE}.  
\subsection{Study of the influence of some general parameters on reconstructions}
\subsubsection{Choice of $\rho_{max}$} First step consists in choosing the maximum diameter for scanning circular arcs $\rho_{max}$. This is equivalent to choosing the maximal scattering angle $\omega_{max}$ since $\omega_{max}=\pi-\arcsin{(R/\rho_{max})}$. We propose in Fig. \ref{fig:varying_rho_max} reconstructions of Derenzo and the bar phantoms for $\rho_{max} = 3000, \,5000$ and $7000$ with a discretization step $\Delta \rho =1$ length unit. These choices correspond respectively to a maximal scattering angle $\omega_{max}= 175, \,177$ and $178$ degrees. A higher $\rho_{max}$ allows to reconstruct the upper right and lower left slopes of straight lines tangent to the ellipses of Shepp-Logan and to the circles of Derenzo and also, small structures of Shepp-Logan. However, one can observe a visual loss of contrast when $\rho_{max}$ increases and this may explain higher NMSE and NMAE values (see Table 1).   

\subsubsection{Choice of $N_\rho$} For the rest of the simulations, $\rho_{max}$ is fixed and equal to $5000$, which seems to be, a good trade off between contrast and good reconstruction of small structures. The objective is now to evaluate the ratio of required data $N_\rho\times N_\varphi$ relative to the number of pixels to reconstruct $N^2$. We denote $Q$ this ratio : $Q = (N_\rho\times N_\varphi)/N^2$. $Q=1$ represents the minimal case where $N_\rho = N_{\rho,min} = N^2/N_\varphi = 163$. This means that, for $Q=1$, $163$ values for $\rho$ are chosen uniformely between $R$ and $\rho_{max}$. Then, the other $N_\rho$ are the product of the chosen ratio $Q$ and $N_{\rho,min}$.  Figure \ref{fig:varying_n_rho} shows the obtained reconstructions for $Q = 1, 5$ and $10$. Although the reconstructions appears to be blurred for $Q = 1$, the algorithm already gives a first estimation of the objects to be reconstructed. 
The \textit{optimal} choice of $Q$ seems to be specific to each object and not directly linked to the proposed algorithm (see Table 1). 

\begin{figure}
    \centering
    \includegraphics{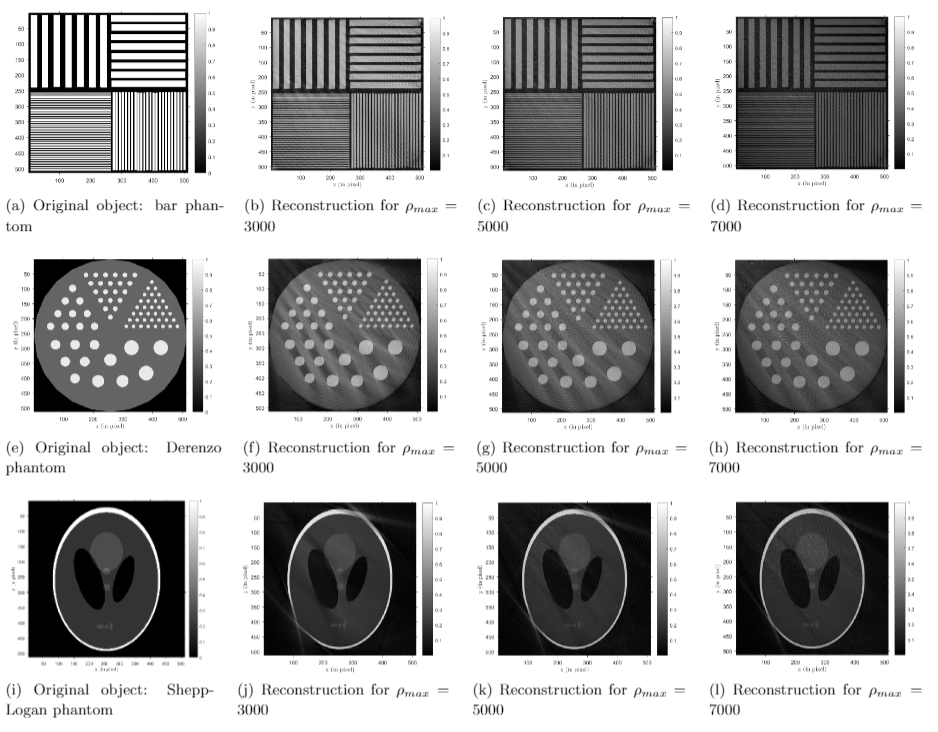}
    \caption{\footnotesize Reconstruction results of the bar (Fig. 3a), Derenzo (Fig. 3b) and Shepp-Logan phantoms (Fig. 3c) for $Q=1$ (a,d,g), $Q=5$ (b,e,h), $Q=10$ (c,f,i)}
    \label{fig:varying_rho_max}
\end{figure}

\begin{figure}[!ht]
    \centering
    \includegraphics{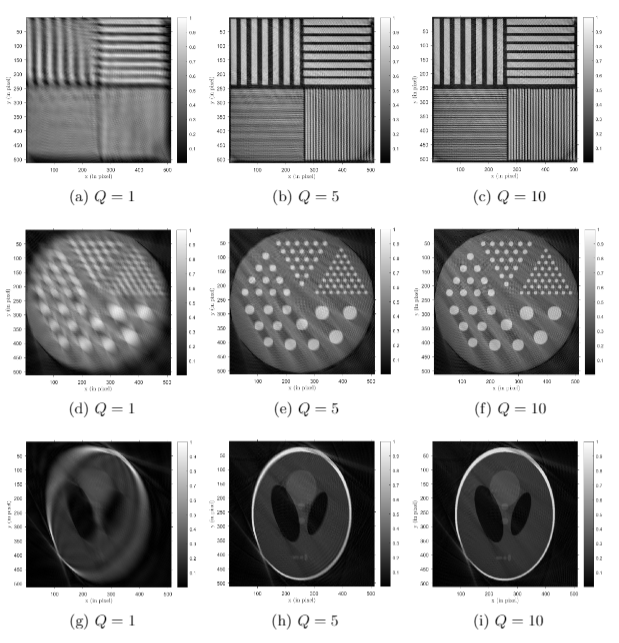}
    \caption{\footnotesize Reconstruction results of the bar (Fig. 3a), Derenzo (Fig. 3b) and Shepp-Logan phantoms (Fig. 3c) for $Q=1$ (a,d,g), $Q=5$ (b,e,h), $Q=10$ (c,f,i)}
    \label{fig:varying_n_rho}
\end{figure}

\subsubsection{Robustness against noise} $Q$ is now fixed and equal to $10$. 
We carry out simulations adding to projections a Gaussian noise of signal-to-noise ratios SNR$= 10, 15$ and $20$ dB. 
 Figure \ref{fig:varying_bruit} shows the obtained results for the three phantoms that exhibit grain artifacts for the higher levels of noise.  

\begin{figure}[!ht]
    \centering
    \includegraphics{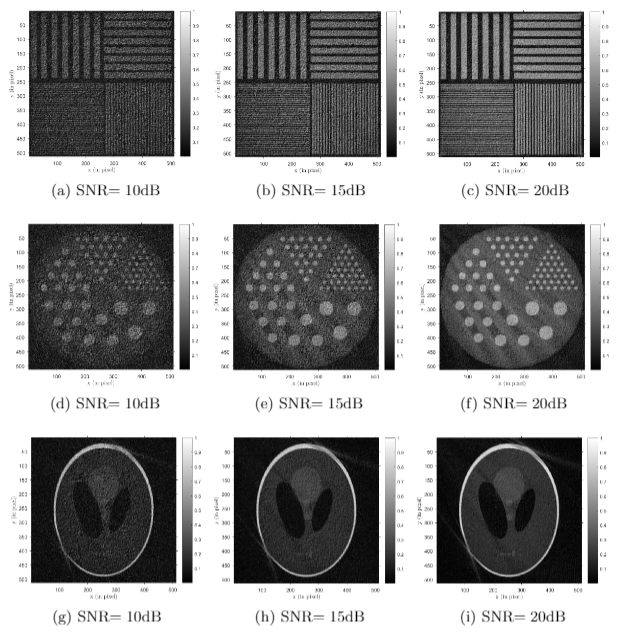}
    \caption{\footnotesize Reconstruction results of the bar (Fig. 3a), Derenzo (Fig. 3b) and Shepp-Logan phantoms (Fig. 3c) with the addition of a Gaussian noise (SNR=$10$dB (a,d,g), SNR=$15$dB (b,e,h), SNR=$20$dB (c,f,i)}
    \label{fig:varying_bruit}
\end{figure}



\begin{table*}
\footnotesize
 \caption{NMSE / NMAE of reconstruction results}
\label{tab:NMSE_NMAE}
\begin{tabularx}{\textwidth}{@{}l|*{3}{C}|*{3}{C}|*{3}{C}}
\toprule
phantoms     & \multicolumn{3}{c}{$\rho_{max}$ (Fig. \ref{fig:varying_rho_max})} & \multicolumn{3}{c}{$Q$ (Fig. \ref{fig:varying_n_rho})} & \multicolumn{3}{c}{SNR (Fig. \ref{fig:varying_bruit})} \\ 
     & $3000$ & $5000$ & $7000$ & $1$ & $5$ & $10$ & $10$dB & $15$dB & $20$dB \\ 
\midrule
Bar    & $0.1073 /$ $0.2866$      & $0.1736/$  $0.3602$   & $0.2769/$  $0.4362$   & $0.2300/$  $0.4504$  & $0.1637/$  $0.3598$   &  $0.1320/$  $0.3190$   & $0.3418/$  $0.4861$     & $0.2601/$  $0.4294$   & $0.1942/$  $0.3783$  \\ 
Derenzo  & $0.0163/$  $0.1019$        & $0.0223/$  $0.1179$        & $0.0319/$  $0.1399$    & $0.0257/$  $0.1258$   & $0.0161/$  $0.0973$    & $0.0159/$  $0.0987$   & $0.0761/$  $0.2235$    & $0.0490/$  $0.1760$   & $0.0287/$  $0.1323$  \\ 
Shepp-Logan & $0.0098/$  $0.0573$  & $0.0098/$  $0.0618$    & $0.0110/$  $0.0652$ & $0.0240/$  $0.0728$ & $0.0121/$  $0.0575$  & $0.0095/$  $0.0550$  & $0.0198/$  $0.0957$   & $0.0140/$  $0.0763$   & $0.0109/$  $0.0621$   \\   
\bottomrule
\end{tabularx}
\end{table*}

\section{Additional considerations and perspectives}

Interesting issues arise when studying the new Compton scattering modality. We briefly detail some of them.

\subsection{Practical considerations}

\subsubsection{Dominance of the Compton effect}
    Compton effect is dominant in matters with low atomic numbers in a wide range of middle-low energies. Particularly, these scenarios include biological tissues and medical applications.
\subsubsection{Assumption of first order scattering}
    All models of CST based on Radon transforms \cite{nguyen10, webber2018three, webber2019microlocal, cebeiro, truong2019compton, tarpau19trpms, tarpau2020compton, cebeiro18, cebeiro2017new} assume that photons undergo scattering only once before reaching detector, i.e. they consider only first order scattering. The consequences that multiple scattering (second order or higher) may have on image quality are still unclear and how to deal with it is an open question. Further work based on Monte Carlo simulations may help to understand this behaviour. In addition, the lower probability and the random nature of this event suggest that noise reduction strategies are a good starting point.
\subsubsection{Inclusion of attenuation in the model} 
   
    An analytic inversion formula for the non-attenuated Radon transform has been shown to play an important role in the reconstruction strategy for the attenuated case. In a practical setting, attenuation and scattering should be considered in the forward model. This leads to a major challenge from the mathematical point of view for which no exact solution is known at present. A number of ways to deal with the problem of attenuation in emission tomography have been studied like the generalized Chang Correction (GCC) \cite{chang1978method}, an iterative algorithm that allows the \textit{a posteriori} correction of the reconstruction or the Iterative Pre Correction algorithm \cite{maze1992iterative} which corrects data. This later method has been adapted for the previous CST Nguyen and Truong's modality in \cite{guerrero2013attenuation}. This study requires having the analytic inverse Radon transform found in the non-attenuated case, i.e. the equivalent problem that we have efficiently solved here for our new CST modality. More recently, the GCC has been adapted by the authors \cite{tarpau2019ndt} for a special bi-modality able to provide the attenuation map and electron density. Results shown in this paper encourage further work on that direction.
\subsubsection{Realistic source and detector} In a practical scenario, source and detector are not punctual but have a finite dimension. Although no study explore the consequences of it, blur can be expected at reconstruction. Furthermore, non ideal energy resolution at detector may also produce some blur due to the fact that uncertainties in the energy of photons translate to uncertainties in the localization of scattering sites (electrons). Nevertheless, permanent progress in detector design has been achieved in recent years and is expected for future years. Regarding the source, angular dependence can be included in a correction factor inside the model as in \cite{rigaud2012novel}.
\subsection{An extension of this modality in three dimensions}
An extension of this modality to three dimensions can be also envisaged: a fixed source and a detector moving on a sphere centered at the source. This setup is even more appealing since no collimator is required neither at the source nor at the detector. This modality leads to a new toric Radon transform, resulting of the rotation of the circular arcs around the source-detector segment. The theoretical challenge raised by this extension consists in the inversion of this new toric Radon transform. Some work on this direction is on the way.

\section{Concluding remarks}
A new modality of Compton Scattering Tomography has been studied in this paper. The design has attractive features such as compactness and simplicity. Furthermore, this CST modality uses uncollimated detectors. This may allow increasing data and simultaneously reducing acquisition time. This modality leads to a new Radon transform on double circular arcs, for which a closed analytic inversion formula is established in this paper. In addition, an efficient reconstruction algorithm based on the Hilbert transform has been developed. Simulation results illustrate the feasibility of this new system and the good performance of the reconstruction algorithm. 


\section{Funding and acknowledgements}
C. Tarpau's research work is supported by grants from R\'egion Ile-de-France (in Mathematics and Innovation) 2018-2021 and LabEx MME-DII (Mod\`eles Math\'ematiques et Economiques de la Dynamique, de l'Incertitude et des Interactions) (No. ANR-11-LBX-0023-01). J. Cebeiro is supported a CONICET postdoctoral grant ($\#$ 171800).

\appendix

\section{Consistency conditions}
We introduce consistency conditions \cite{cormack84, truong2011radon} in terms of Cormack sense in order to deduce a closed formulation of \eqref{inv_gn}, more suitable for numerical computation. Equation (\ref{inv_gn}) can be rewritten using the n-th order Tchebychev polynomial of the first kind $T_n$  as 
\begin{equation}
        f_n(r) = \frac{1}{\pi} \frac{d}{dr}\int_R^r   \frac{T_{|n|}\left(r/\rho\right)}{\rho\sqrt{\left(r/\rho\right)^2-1}}\, \frac{(\mathcal{R}_{\mathcal{D}}f)_{n}(\rho)}{ 2\cos{\left(n\cos^{-1}(R/\rho)\right)}}d\rho
        \label{inv_gn_tch}, 
    \end{equation}
    
Using the following relationship between Tchebychev polynomials of the first kind $T_n$ and the second kind $U_n$ \cite{cormack84}
\begin{equation}
    \frac{T_{|n|}\left(r/\rho\right)}{\sqrt{\left(r/\rho\right)^2-1}}=\frac{\left(\left(r/\rho\right)-\sqrt{\left(r/\rho\right)^2-1}\right)^{|n|}}{\sqrt{\left(r/\rho\right)^2-1}} + U_{|n|-1}\left(r/\rho\right), 
     \label{eq:relation_tcheby}
\end{equation}

\noindent and according to (\ref{eq:Gn}), (\ref{inv_gn_tch}) becomes

\begin{equation}
   f_n(r) = \frac{1}{\pi} \frac{d}{dr} \left[ \int_R^r  \frac{\left(\left(r/\rho\right)-\sqrt{\left(r/\rho\right)^2-1}\right)^{|n|}}{\sqrt{\left(r/\rho\right)^2-1}}\,G_n(\rho)\frac{d\rho}{\rho} +\int_R^r U_{|n|-1}\left(r/\rho\right) \,G_n(\rho)\frac{d\rho}{\rho}\right].
        \label{inv_gn_tchU}   
\end{equation}

Furthermore, from (\ref{eq:g_n}), and changing the order of integration, one can obtain  for $n\in\mathbb{N}^\star$ and $k=n, n-2, n-4, ... >0$ 


\begin{equation}
    \int_R^\infty  G_n(\rho)\frac{d\rho}{\rho^k} =2\int_R^\infty  f_n(r)\int_r^\infty \frac{\cos{\left(n\cos^{-1}\left(r/\rho\right)\right)}}{\sqrt{1-\left(r/\rho\right)^2}}\frac{d\rho}{\rho^k}dr\label{eq:cons_init_change_order}. 
\end{equation}

With the change of variables $\rho = r/\cos{\nu}$, the previous $\rho$-integral becomes 
\begin{equation}
    \int_r^\infty \frac{\cos{\left(n\cos^{-1}\left(r/\rho\right)\right)}}{\sqrt{1-\left(r/\rho\right)^2}}\,\frac{d\rho}{\rho^k} =r^{-k+1}\int_0^{\frac{\pi}{2}} \frac{\cos{(n\nu)}}{\cos{(\nu)}^{2-k}}d\nu.
\end{equation}

We arrive to an $\nu$-integral whose result is given in \cite{magnus2013formulas}
\begin{equation}
    \int_0^{\frac{\pi}{2}} \frac{\cos{(n\nu)}}{\cos{(\nu)}^{2-k}}d\nu = \frac{\pi}{(k-1)2^{k-1}}\frac{\Gamma(k)}{\Gamma\left(\frac{k+n}{2}\right)\Gamma\left(\frac{k-n}{2}\right)},
\end{equation}
where $\Gamma$ refers to the gamma function. For $n>0$ and denoting $h(n)=\Gamma\left(\frac{k+n}{2}\right)\Gamma\left(\frac{k-n}{2}\right)$, one can remark that $k=n, n-2, n-4>0$ are poles of $h$. Since $h$ is even, then $k=-n, -n-2, -n-4<0$ for $n<0$ are also poles of $h$. Consequently, one can conclude that this integral vanishes \cite{truong2011radon}, i.e., for $n\in\mathbb{N}^\star$ and $k=n, n-2, n-4, ... >0$, 
    \begin{equation}
        \int_R^\infty \frac{(\mathcal{R}_{\mathcal{D}}f)_{n}(\rho)}{2 \cos{(n\psi)}}\frac{d\rho}{\rho^k} =0.\label{eq:vanish}
    \end{equation}

Furthermore, since Tchebychev polynomials of second kind $U_{|n|-1}(r/\rho)$ are a linear sum of polynomials $1/\rho^k$ (see \cite{magnus2013formulas}), \eqref{eq:vanish} leads to 
    \begin{equation}
        \int_R^\infty  U_{|n|-1}\left(\frac{r}{\rho}\right)\,\frac{(\mathcal{R}_{\mathcal{D}}f)_{n}(\rho)}{2 \cos{(n\psi)}} d\rho =0, n\in \mathbb{N}^\star.\label{eq:vanishU}
    \end{equation}

\section{Derivation of formula \eqref{eq:recons_f}}

This section uses consistency conditions established in \eqref{eq:vanishU} to obtain the final closed-form inversion formula \eqref{eq:recons_f}. From \eqref{eq:vanishU}, we deduce $\forall r\in [R, +\infty[,$
\begin{equation}
    \int_R^r  U_{|n|-1}\left(\frac{r}{\rho}\right)\,\frac{(\mathcal{R}_{\mathcal{D}}f)_{n}(\rho)}{2 \cos{(n\psi)}}d\rho= -\int_r^\infty \, U_{|n|-1}\left(\frac{r}{\rho}\right)\,\frac{(\mathcal{R}_{\mathcal{D}}f)_{n}(\rho)}{2 \cos{(n\psi)}}d\rho.\label{eq:consequence_vanishU}
\end{equation}

Consequently, using \eqref{eq:consequence_vanishU} and the relation between Tchebychev polynomials of the first and the second kind \eqref{eq:relation_tcheby}, we obtain 
\begin{equation}
        f_n(r) =\frac{1}{\pi} \frac{d}{dr}\left[ \int_R^r  \frac{\left(\left(r/\rho\right)-\sqrt{\left(r/\rho\right)^2-1}\right)^{|n|}}{\sqrt{\left(r/\rho\right)^2-1}}\frac{(\mathcal{R}_{\mathcal{D}}f)_{n}(\rho)}{2 \cos{(n\psi)}}\frac{d\rho}{\rho}-\int_r^\infty U_{|n|-1}\left(r/\rho\right)\frac{(\mathcal{R}_{\mathcal{D}}f)_{n}(\rho)}{2 \cos{(n\psi)}}\frac{d\rho}{\rho}\right]
        \label{eq:inv_stable}.
    \end{equation}
    
Then, with $f$ and $G$ the projections of circular harmonic components of $G_n$ and $f_n$ in (\ref{eq:inv_stable}), we obtain
\begin{multline}
    f(r,\theta) = \frac{1}{\pi}\frac{d}{dr} \left[ \int_R^r  \int_0^{2\pi} \frac{1}{2\pi\rho}\frac{G(\rho,\varphi)}{\sqrt{\left(r/\rho\right)^2-1}}\, \cdot\right.\sum_{-\infty}^{\infty}\left(\left(r/\rho\right)-\sqrt{\left(r/\rho\right)^2-1}\right)^{|n|} e^{in(\theta-\varphi)}d\varphi d\rho-\\ \left.\int_r^\infty \frac{1}{2\pi\rho}\int_0^{2\pi} G(\rho, \varphi)\, \sum_{-\infty}^{\infty} U_{|n|-1}\left(r/\rho\right) \,e^{in(\theta-\varphi)}d\varphi d\rho\right].
    \label{eq:previous_lemma}
\end{multline}

We need the two following lemmas to simplify \eqref{eq:previous_lemma}.
\begin{lemma}[\cite{truong2011radon}, Lemma 3.8]
For $1<\rho<r$, we have 
\begin{equation}
    1+2\sum_{n=1}^\infty\left(r/\rho-\sqrt{\left(r/\rho\right)^2-1}\right)^n\cos{(n(\theta-\varphi))}=\frac{-\sqrt{\left(r/\rho\right)^2-1}}{\cos{(\theta-\varphi)}-(r/\rho)}
\end{equation}
\end{lemma}

\begin{lemma}[\cite{truong2011radon}, Lemma 3.9]For $0<r<\rho$, we have
\begin{equation}
    \sum_{n\in\mathbb{Z}} U_{|n|-1}\left(r/\rho\right)e^{in(\theta-\varphi)} = \frac{1}{\cos{(\theta-\varphi)-(r/\rho)}}
\end{equation}

\end{lemma}

Hence, we obtain :
\begin{equation}
    f(r,\theta) =\\\frac{1}{2\pi^2} \int_0^{2\pi}\frac{d}{dr}\,\hbox{p.v.}\left[\int_R^\infty  \frac{ G(\rho, \varphi)}{r-\rho\cos{(\theta-\varphi)}}d\rho\right]d\varphi.
    \label{eq:previous}
\end{equation}
Then, \eqref{eq:recons_f} is obtained first inverting the order of the $r-$derivative and the $\rho$-integral and then with the substitution $u=\rho/r$ and the fact that $G(\rho, \varphi)=0$ for $\rho\in[0,R[$. 



\begin{remark}
The intermediate result \eqref{eq:inv_stable} is also an inversion formula based on circular harmonic expansion. However, the proposed formulation \eqref{eq:recons_f} is more practical for numerical reconstruction than \eqref{eq:inv_stable}. In fact, \eqref{eq:inv_stable} can be numerically implemented using a method similar to Chapman and Cary numerical approach \cite{chapman1986circular}, see for instance \cite{num_inv}. Nevertheless, the technique requires the evaluation of Tchebychev functions in \eqref{eq:inv_stable} in terms of primitive integrals that are evaluated recursively. This implies either longer computational time or more memory.
\end{remark}

\bibliographystyle{IEEEtran}
{\footnotesize
\bibliography{IEEEabrv,biblio}}
\end{document}